\documentclass[11pt,a4paper]{article}
\usepackage{a4wide,amsfonts,amsmath,latexsym,amsthm,amssymb,euscript,eufrak,graphicx,units,mathrsfs,setspace}

\usepackage{float}

\usepackage[francais,english]{babel}

\newtheorem{prop}{Proposition}[section]
\newtheorem{remark}[prop]{Remark}
\newtheorem{theorem}[prop]{Theorem}
\newtheorem{lemma}[prop]{Lemma}

\newtheorem{definition}[prop]{Definition}

\author{Monique Jeanblanc\footnote{Universit\'e dÕEvry Val dÕEssonne, Laboratoire Analyse et Probabilit\'es,
 \texttt{monique.jeanblanc@univ-evry.fr}} \and Anthony R\'eveillac\footnote{CEREMADE UMR CNRS 7534, Universit\'e Paris Dauphine, \texttt{anthony.reveillac@ceremade.dauphine.fr}}}

\title{A Note on BSDEs with singular driver coefficients}

\begin{document}

\maketitle

\begin{abstract}
In this note we study a class of BSDEs which admits a particular singularity in the driver. More precisely, we assume that the driver is not integrable and degenerates when approaching to the terminal time of the equation.
\end{abstract}

\section{Introduction}

Since the seminal works of Bismut \cite{mjarBismut_78} and of Pardoux and Peng \cite{mjarPardoux_Peng90}, a lot of attention has been given to the study of Backward Stochastic Differential Equations (BSDEs) as this object naturally arises in stochastic control problems and was  found to be an ad hoc tool for many financial applications as illustrated in the famous guideline paper \cite{mjarElKaroui_Peng_Quenez}. Recall that a BSDE takes the following form:
$$ Y_t=Y_T -\int_t^T f(s,Y_s,Z_s) ds -\int_t^T Z_s dW_s, \quad t\in [0,T],$$
where $W$ is a multi-dimensional Brownian motion. The historical natural assumption for providing existence and uniqueness (in the appropriate spaces) is to assume the driver $f$ to be Lipschitz plus some integrability conditions on the terminal condition. However, in applications one may deal with drivers which are not Lipschitz continuous, and which exhibit e.g. a quadratic growth in $z$ (in the context of incomplete markets in Finance), or only some monotonicity in the $y$-variable. One way of relaxing the Lipschitz growth condition in $y$ is the so-called stochastic Lipschitz assumption which basically consists in replacing the usual Lipschitz constant by a stochastic process satisfying appropriate integrability conditions. As noted in Section 2.1.2 "Pathology" in \cite{mjarElKaroui_redbook}, even in the stochastic linear framework, one has to be very careful when relaxing the integrability conditions on the driver of the equation. As an illustration, consider the following example presented in \cite{mjarElKaroui_redbook} (cf. \cite[(2.9)]{mjarElKaroui_redbook}):
\begin{equation}
\label{mjareq:EKred}
Y_t =0+\int_t^T [r Y_s +\sigma Z_s + \gamma Y_s (e^{\gamma (T-s)}-1)^{-1}] ds +\int_t^T Z_s dW_s, \quad t\in [0,T],
\end{equation}
where $W$ is a one-dimensional Brownian motion, $r, \sigma, \gamma>0$ and $T$ is a fixed positive real number. It is proved in \cite{mjarElKaroui_redbook} that the BSDE \eqref{mjareq:EKred} has an infinite number of solutions. Note that here the driver is not Lipschitz continuous in $y$ due to the exploding term $(e^{\gamma (T-t)}-1)^{-1}$ as $t$ goes to $T$, and completely escapes the existing results of the literature.\\\\
\noindent
The aim of this note is to elaborate on the pathology mentioned in \cite{mjarElKaroui_redbook} and to try to understand better what kind of behavior can appear as soon as the usual integrability conditions are relaxed. In light of Example \eqref{mjareq:EKred}, multiple solutions is one of the behaviour which can be observed. However, is it the only type of problem that can occur ? For instance is it clear that existence is guaranteed? This note is an attempt in this direction and is motivated by the work in preparation \cite{mjarJNR} where equations with this specific pathology appear naturally in the financial application under interest in \cite{mjarJNR}.\\\\
\noindent
We proceed as follows. First we make precise the context of our study and we explain what is the notion of solution we use for dealing with non-integrable drivers. Then we deal with the particular case of affine equations in Section \ref{mjarsection:affine}. These equations already allow us to present several type of pathologic behaviour. We then study in Section \ref{mjarsection:nonlinear} a class of non-linear drivers which will be of interest for a specific financial application presented in \cite{mjarJNR}.  In particular, in our main result Theorem \ref{mjarth:expo2} we provide an existence and uniqueness result under a monotonicty assumption on the mapping $f$ in \eqref{mjareq:proto} defined below.

\section{Preliminaries}

In this note $T$ denotes a fixed positive real number and $d$ a given positive integer. We set $(W_t)_{t\in [0,T]}:=(W_t^1,\ldots,W_t^d)_{t\in [0,T]}$ a $d$-dimensional standard Brownian motion defined on a filtered probability space $(\Omega,\mathcal{F}, \mathbb{F}:=(\mathcal{F}_t)_{t\in [0,T]},\mathbb{P})$ where $\mathbb{F}$ denotes the natural filtration of $W$ (completed and right-continuous) and $\mathcal{F}=\mathcal{F}_T$. Throughout this paper "$\mathbb{F}$-predictable" (rep. $\mathbb{F}$-adapted) processes will be referred to predictable (resp. adapted) processes. For later use we set for $p\geq 1$:
$$ \mathbb{S}^{p}:=\left\{(Y_t)_{t\in [0,T]} \textrm{ continuous adapted one dimensional process }, \, \mathbb{E}\left[\sup_{t\in [0,T]} |Y_t|^p\right]<+\infty \right\}, $$
$$ \mathbb{H}^{p}(\mathbb{R}^m):=\left\{(Z_t)_{t\in [0,T]} \textrm{ predictable } m\textrm{-dimensional process }, \, \mathbb{E}\left[\left(\int_0^T \|Z_t\|^2 dt\right)^{p/2}\right]<+\infty \right\},$$
where $\|\cdot\|$ denotes the Euclidian norm on $\mathbb{R}^m$ ($m\geq 1$). For any element $Z$ of $\mathbb{H}^1(\mathbb{R}^d)$, we set $\int_0^\cdot Z_s dW_s:=\sum_{i=1}^d \int_0^\cdot Z_s^i dW_s^i$. We also set $L^p:=L^p(\Omega,\mathcal{F}_T,\mathbb{P})$.\\\\
\noindent
Let $\lambda:(\lambda_t)_{t\in [0,T]}$ be a one-dimensional non-negative predictable process. For convenience we set $\Lambda_t:=\int_0^t \lambda_s ds, \; t\in [0,T].$ We make the following \\\\
\textbf{Standing assumption on $\lambda$}:
$$ \Lambda_t<+\infty, \; \forall t<T, \; \textrm{ and } \Lambda_T=+\infty, \; \mathbb{P}-a.s.$$
The typical example we have in mind is a coefficient $\lambda$ of the form $\lambda_t:=(e^{\gamma (T-t)}-1)^{-1}$ as in the introducing example \eqref{mjareq:EKred}, or when $\lambda$ is the intensity process related to a prescribed random time $\tau$ in the context of enlargement of filtration as presented in \cite{mjarJNR}. In this note, we aim in studying BSDEs of the form:
\begin{equation}
\label{mjareq:proto}
Y_t=A-\int_t^T [\varphi_s + \lambda_s f(Y_s)] ds -\int_t^T Z_s dW_s, \quad t\in [0,T],
\end{equation}
where $A$ is a regular enough $\mathcal{F}_T$-measurable random variable, $f:\mathbb{R} \to \mathbb{R}$ is a deterministic map and $\varphi$ is a predictable processes with some integrability conditions to be specified. Before going further, we would like to stress that in contradistinction to the classical case where $\lambda$ is bounded (and $A$ and $\varphi$ are square-integrable), the space $\mathbb{S}^{2}\times \mathbb{H}^{2}(\mathbb{R}^d)$ is no more the natural space for solutions of our BSDEs. For instance if $f(x):=x$, the fact that $(Y,Z)$ belongs to $\mathbb{S}^{2}\times \mathbb{H}^{2}(\mathbb{R}^d)$ does not guarantee that
$$ \mathbb{E}\left[ \int_0^T |\lambda_s Y_s|^p ds\right]<+\infty$$
for some $p\geq 1$ (which would be  immediately satisfied with $p=2$ if $\lambda$ were bounded) leading to a possible definition problem for the term $\int_0^t \lambda_s Y_s ds$ in equation \eqref{mjareq:proto}. For this reason we make very precise the notion of solution in our context.

\begin{definition}[Solution]
\label{mjarsolution}
Let $A$ be an element of $L^1$ and $f:\Omega\times [0,T] \times \mathbb{R} \times \mathbb{R}^d \to \mathbb{R}$ such that for any $(y,z)$ in $\mathbb{R}\times \mathbb{R}^d$ the stochastic process $(t,\omega) \mapsto f(t,y,z)$ (where as usual we omit the $\omega$-variable in the expression of $f$) is progressively measurable. We say that a pair of predictable processes $(Y,Z)$ with values in $\mathbb{R}\times\mathbb{R}^d$ is a solution to the BSDE
\begin{equation}
\label{mjareq:BSDEtheo1}
Y_t=A-\int_t^T f(s,Y_s,Z_s) ds -\int_t^T Z_s dW_s, \quad t\in [0,T],
\end{equation}
if
\begin{equation}
\label{mjareq:BSDEtheo2}
\mathbb{E}\left[\int_0^T |f(t,Y_t,Z_t)| dt + \left(\int_0^T \|Z_t\|^2 dt\right)^{1/2}\right]<+\infty,
\end{equation}
and
Relation \eqref{mjareq:BSDEtheo1} is satisfied for any $t$ in $[0,T]$, $\mathbb{P}$-a.s.
\end{definition}

\begin{remark}
\label{mjarrk:D}
This notion of solution is related to the theory of $L^1$-solution (see e.g. \cite[Definition 2.1]{mjarBriandDelyonHuPardouxStoica} or \cite{mjarConfortolaFuhrmanJacod,mjarmjarConfortolaFuhrmanJacod_multijumps}) where in Relation \eqref{mjareq:BSDEtheo2} the expectation is replaced by a $\mathbb{P}$-a.s. criterion. The fact that $Z$ is an element of $\mathbb{H}^1(\mathbb{R}^d)$ implies that the martingale $\int_0^\cdot Z_s dW_s$ is uniformly integrable. Combining this property with the $(\Omega \times [0,T],\mathbb{P}\otimes dt)$-integrability of $f(\cdot,Y,Z)$, it immediately follows that the solution process $Y$ is of class (D) (which then finds similarities with the notion of solution used in \cite{mjarBriandDelyonHuPardouxStoica}).
\end{remark}

\begin{remark}
We would like to stress that even in the case where the terminal condition $A$ is in $L^2$ we do not require $Y$ to be an element of $\mathbb{S}^2$. This fact bears some similarities with the papers \cite{mjarConfortolaFuhrmanJacod,mjarmjarConfortolaFuhrmanJacod_multijumps} and with \cite[Section 6]{mjarBriandDelyonHuPardouxStoica}.
\end{remark}

\begin{remark}[Classical $L^2$ setting]
If $f$ is uniformly (in time) Lipschitz in $(y,z)$ and if $\mathbb{E}\left[|A|^2+\int_0^T |f(s,0,0)|^2 ds\right]<+\infty$,  then the fact that there exists $(Y,Z)$ in $\mathbb{S}^{2}\times \mathbb{H}^{2}(\mathbb{R}^d)$ satisfying \eqref{mjareq:BSDEtheo1} implies that the process $f(\cdot,Y_\cdot,Z_\cdot)$ is in $\mathbb{H}^2(\mathbb{R}^d)$ and thus Relation \eqref{mjareq:BSDEtheo2} is satisfied.
\end{remark}

Another important issue in our context is uniqueness. The uniqueness for the $Z$ component will be understood in the $\mathbb{H}^1(\mathbb{R}^d)$ sense. Concerning the $Y$ component, since we do not impose $Y$ to belong to $\mathbb{S}^1$ we will say that $Y^1=Y^2$ if the processes are indistinguishable (by definition of a solution, both processes are continuous, and hence uniqueness boils down to require $Y^1$ to be a modification of $Y^2$). This definition for uniqueness in our very special setting coincides with the notion of uniqueness with respect to a particular norm. More precisely, according to Remark \ref{mjarrk:D} a solution process $Y$ is of class (D). This space can be naturally equipped with the norm $\|\cdot \|_{(D)}$ defined as\footnote{This norm is referred as $\|\cdot\|_1$ in \cite[Definition VI.20]{mjarDellacherieMeyer_2}, we do not use this notation here to avoid any confusion.}:
$$ \|X\|_{(D)}:=\sup_{\tau \in \cal T} \mathbb{E}[|X_\tau|], \quad X \textrm{ of class } (D),$$
where $\mathcal{T}$ denotes the set of stopping time smaller or equal to $T$. By \cite[Theorem IV.86]{mjarDellacherieMeyer_1}, uniqueness with respect to the norm $\|\cdot \|_{(D)}$ is equivalent to indistinguishability.
\\\\
From now on, by solution to a BSDE we mean a solution in the sense of Definition \ref{mjarsolution}. For any pair of ($\mathcal{F}_T$-measurable) random variables $(A,B)$, we write $A\not\equiv B$ if $\mathbb{P}[A \neq B]>0$. Similarly, $A=B$, $\mathbb{P}$-a.s. will be denoted as $A \equiv B$. Throughout this paper $C$ will denote a generic constant which can differ from line to line.

\section{Affine equations with exploding coefficients}
\label{mjarsection:affine}

As the reader will figure out later, it seems pretty complicated to define a general theory since many situations (non-existence, non-uniqueness) can be found under our assumption on $\lambda$ for BSDEs of the form \eqref{mjareq:proto}. These very different behaviours can be clearly illustrated by studying affine equations, that is when $f$ in \eqref{mjareq:proto} stands for the identity (or minus the identity). In some sense, our results find immediate counterparts in the deterministic realm while considering the corresponding ODEs when all the coefficients of the equation are deterministic. However, for this latter case, techniques of time-reversion can be employed to provide immediate results which unfortunately can not be applied in the stochastic framework due to the measurability feature of the solution to a BSDE calling for different techniques.\\\\
\noindent
In this section, we consider stochastic affine BSDEs of one of the following forms:

\begin{equation}
\label{mjareq:l1}
Y_t=A- \int_t^T (\varphi_s -\lambda_s Y_s) ds -\int_t^T Z_s dW_s; \quad t\in [0,T],
\end{equation}

\begin{equation}
\label{mjareq:l2}
Y_t=A- \int_t^T (\varphi_s +\lambda_s Y_s) ds -\int_t^T Z_s dW_s; \quad t\in [0,T],
\end{equation}

\noindent
We start with Equation \eqref{mjareq:l1}.
\begin{prop}
\label{mjarprop:aff-}
Let $A$ be in $L^1$ and $\varphi:=(\varphi_t)_{t\in [0,T]}$ be  an element of $\mathbb{H}^1(\mathbb{R})$. The Brownian BSDE
\begin{equation}
\label{mjareq:aff-}
dY_t= (\varphi_t-\lambda_t Y_t) dt + Z_t dW_t; \quad Y_T=A.
\end{equation}
admits no solution if $A\not\equiv
0$. If $A\equiv 0$, the BSDE \eqref{mjareq:aff-} may admit infinitely many solutions.
\end{prop}

\begin{proof}
\textbf{Step 1: non-existence of solution if $A\not\equiv 0$\\\\}
Let $(Y,Z)$ be a solution to \eqref{mjareq:aff-}. Assume there exists a set $\mathcal{A}$ in $\mathcal{F}_T$ such that $A>0$ on $\mathcal{A}$. By definition of a solution, it holds that
\begin{equation}
\label{mjareq:tempaff}
\int_0^T |\lambda_s Y_s| ds <\infty, \quad \mathbb{P}\textrm{-a.s.}
\end{equation}
since $$\int_0^T |\lambda_s Y_s| ds \leq \int_0^T |\varphi_s - \lambda_s Y_s| ds+\int_0^T |\varphi_s| ds.$$
For $\omega$ in $\cal A$, let $t_0(\omega):=\sup \{t\in [0,T], \; Y_t(\omega) <A/2\}$. By continuity of $Y$ and the fact that $Y_T=A$, for $\mathbb{P}$-almost all $\omega$ in $\mathcal{A}$, $t_0(\omega)<T$ and $Y_t(\omega) \textbf{1}_{[t_0(\omega),T]}(t) \geq A/2$. Note that $t_0$ is not a stopping time but only a $\mathcal{F}_T$-measurable random variable. As a consequence, on $\cal A$, it holds that
$$ \int_0^T | \lambda_s Y_s| ds \geq  \int_{t_0}^T \lambda_s Y_s ds \geq A/2 \underset{=+\infty}{\underbrace{\int_{t_0}^T \lambda_s ds}},$$
which contradicts \eqref{mjareq:tempaff}.
As a consequence, $A\leq 0$, $\mathbb{P}$-a.s.. Similarly, one proves that $A\geq 0$, $\mathbb{P}$-a.s..\\\\
\textbf{Step 2: Multiplicity of solutions if $A \equiv 0$\\\\}
If $A\equiv 0$, we will provide examples of non-uniqueness of solution. Remark that if $(\mathcal{Y}, \mathcal{Z})$ is a particular solution to the BSDE and that $(Y,Z)$ is a solution to the (fundamental) BSDE:
\begin{equation}
\label{mjareq:fond-}
dY_t =-\lambda_t Y_t dt + Z_t dW_t, \quad Y_T=0,
\end{equation}
then as for ODEs, the sum of any of these fundamental solutions and $\mathcal{Y}$ is a solution to \eqref{mjareq:aff-} (together with the sum of the associated $Z$ processes).
In addition, Equation \eqref{mjareq:fond-} admits an infinite number of solutions (like $Y_t=Y_0 e^{-\Lambda_t}$ and $Z\equiv 0$ which is an adapted continuous solution to the BSDE satisfying $\mathbb{E}\left[\int_0^T \lambda_s |Y_s| ds\right]=|Y_0|$ for any chosen real number $Y_0$).
An example of particular solution can be given by the process $\mathcal{Y}_t:=- \mathbb{E}\left[\int_t^{T} e^{\int_t^s \lambda_u du} \varphi_s
ds\vert \mathcal{F}_t \right]$ if it is well-defined, such that Relation \eqref{mjareq:BSDEtheo2} is satisfied. In that case, the existence of $\mathbb{E}\left[\int_t^T \varphi_s e^{\Lambda_s } ds \vert \mathcal{F}_t\right]$ entails that it converges to $0$ as $t$ goes to $T$, and hence that $\mathcal{Y}_T=0$. One can check that $\mathcal{Y}$ together with the process $\mathcal{Z}:=\tilde Z
e^{-\Lambda}$ is solution to \eqref{mjareq:aff-}, where $\tilde Z$ is such that $ \mathbb{E} \left[\int_0^T
\varphi_s e^{\Lambda_s} ds \vert \mathcal{F}_t\right] =
\mathbb{E}\left[ \int_0^T \varphi_s e^{\Lambda_s} ds\right] -\int_0^t
\tilde Z_s dW_s$ ($t\in [0,T]$).
We conclude the proof with an example: set $\varphi_t:=e^{-\Lambda_t}$. With this choice, the process $\mathcal{Y}$ satisfies all the requirements above providing infinitely many solutions to \eqref{mjareq:aff-}.
\end{proof}

Note that in the previous proof the non-existence when $A\not\equiv 0$ relies on the assumption that $\int_0^T |\varphi_s| ds<+\infty, \;\mathbb{P}-a.s.$ If the latter is not satisfied, one may find existence of solutions for $A \not\equiv 0$ as the following proposition illustrates in the deterministic case.
\begin{prop}
\label{mjarprop:ODE}
\label{mjarprop:affdet-} Let $A$ be a given constant and
$\varphi:=(\varphi_t)_{t\in [0,T]}$ be a deterministic map. We
assume that $\lambda$ is a deterministic function such that
$\Lambda_t=\int_0^t \lambda _sds<+\infty$, for $t<T$, and  $\int_0^T
\lambda _sds=+\infty$. Then
\begin{itemize}
\item[(i)] If $e^{-\Lambda_t} \int_0^t e^{\Lambda _s}\varphi _s ds$ converges
to $C$ when $t$ goes to $T$, then the ODE
\begin{equation}
\label{mjareq:affdet-} dY_t= (\varphi_t-\lambda_t Y_t) dt  ; \quad
Y_T=A.
\end{equation}
admits no solution if $A\neq C$. If $A=C$, the ODE
\eqref{mjareq:affdet-} admits infinitely many solutions given by $Y_t=e^{-\Lambda_t} \left(Y_0+
\int_0^t e^{\Lambda _s}\varphi _s ds\right)$ provided that $\int_0^T |\varphi_t - \lambda_t Y_t| dt<\infty$.
\item[(ii)] If
$e^{-\Lambda_t} \int_0^t e^{\Lambda _s}\varphi _s ds$ does not
converge, the ODE \eqref{mjareq:affdet-} has no solution.
\end{itemize}
 \end{prop}
\begin{remark}
\label{mjarrk:34}
Note that the assumption in (i) of Proposition \ref{mjarprop:affdet-} when $C=0$, can be met only if $\int_0^T |\varphi_s| ds=+\infty$. Indeed, assume that $\int_0^T |\varphi_s| ds<\infty$. Let $\varepsilon>0$ and $t<T$. We have:
\begin{align*}
e^{-\Lambda_t} \left| \int_0^t e^{\Lambda_s} \varphi_s ds \right|&\leq e^{-\Lambda_t} \int_0^{t-\varepsilon} e^{\Lambda_s} |\varphi_s| ds +  e^{-\Lambda_t} \int_{t-\varepsilon}^t e^{\Lambda_s} |\varphi_s| ds\\
&\leq e^{-\Lambda_t} e^{\Lambda_{t-\varepsilon}} \int_0^T |\varphi_s| ds + \int_{t-\varepsilon}^t |\varphi_s| ds.
\end{align*}
Hence as $t$ goes to $T$, we have that
$ \lim_{t\to T} e^{-\Lambda_t} \left|\int_0^t e^{\Lambda_s} \varphi_s ds\right| \leq \int_{T-\varepsilon}^T |\varphi_s| ds,$
and hence
$$ \lim_{t\to T} e^{-\Lambda_t} \left|\int_0^t e^{\Lambda_s} \varphi_s ds\right| =0,$$
which contradicts the assumption of (i).
\end{remark}

\begin{remark}
Since $\lambda$ is unbounded, assuming $A, \lambda$ and $\varphi$ to be deterministic in Equation \eqref{mjareq:l1} does not lead to deterministic solutions (and so differs from the ODE framework of Proposition \ref{mjarprop:ODE}) as the following example illustrates. Assume $A\equiv 0$, $\varphi\equiv 0$ and $\lambda$ is a deterministic mapping. Then for any element $\beta:=(\beta_t)_{t\in [0,T]}$ in $\mathbb{H}^1(\mathbb{R}^d)$, the pair of adapted processes $(Y,Z)$ defined as:
$$ Y_t:=Y_0 e^{-\Lambda_t} + e^{-\Lambda_t} \int_0^t \beta_s dW_s, \quad Z_t:=e^{-\Lambda_t} \beta_t,\quad t\in [0,T]$$
is a solution to \eqref{mjareq:l1}. This provides in turn a generalization of the fundamental solution to Equation \eqref{mjareq:fond-}.
\end{remark}

We continue with the BSDE:
$$ dY_t= (\varphi_t+\lambda_t Y_t) dt + Z_t dW_t; \quad Y_T=A.$$

\begin{prop}
\label{mjarprop:affine1}
Let $A$ be in $L^1$ and $\varphi:=(\varphi_t)_{t\in[0,T]}$ be a bounded predictable process. The Brownian BSDE
\begin{equation}
\label{mjareq:affine1}
dY_t= (\varphi_t + \lambda_t Y_t) dt + Z_t dW_t; \quad Y_T=A.
\end{equation}
admits no solution unless $A\equiv 0$. If $A \equiv 0$, then the BSDE admits a unique solution.
\end{prop}

\begin{proof}
Let $(Y,Z)$ be a solution and set $\tilde{Y}:=Y e^{-\Lambda}-\int_0^\cdot e^{-\Lambda_s} \varphi_s ds$. We have that
$$d\tilde{Y}_t = e^{-\Lambda_t} Z_t dW_t, \; \textrm{ and } \tilde{Y}_T=-\int_0^T e^{-\Lambda_s} \varphi_s ds. $$
Hence $\tilde{Y}$ is a $L^1$-martingale and
$$ \tilde{Y}_t =-\mathbb{E}\left[\int_0^T e^{-\Lambda_s} \varphi_s ds \vert \mathcal{F}_t\right], \quad t\in [0,T],$$
leading to
\begin{equation}
\label{mjareq:Yfacil}
Y_t=-\mathbb{E}\left[\int_t^T e^{-\int_t^s \lambda_u du} \varphi_s ds \vert \mathcal{F}_t\right], \quad t\in [0,T].
\end{equation}
In particular, $Y_T=0$. Indeed, since $\varphi$ is bounded
$$e^{ \Lambda_t}\left\vert \int_t^T e^{-\Lambda_s} \varphi_s
ds\right\vert \leq e^{ \Lambda_t} e^{-\Lambda_t} \int_t^T |\varphi_s| ds\leq \|\varphi\|_\infty (T-t).$$
This proves
that there is no solution to the equation unless $A\equiv 0$. We now assume that $A \equiv 0$. In that case, we prove that the process given by
\eqref{mjareq:Yfacil} together with a suitable process $Z$ is a solution to the BSDE. We begin with the integrability condition
$$ \mathbb{E}\left[\int_0^T \vert \lambda_s Y_s \vert ds\right] <+\infty.$$
We have
\begin{align*}
\mathbb{E}\left[\int_0^T |\lambda_s Y_s| ds\right]&=\mathbb{E}\left[\int_0^T \left|\lambda_s \mathbb{E}\left[\int_s^T e^{-\int_s^u \lambda_r dr} \varphi_u du \vert \mathcal{F}_s\right] \right| ds\right]\\
&\leq \mathbb{E}\left[\int_0^T \lambda_s e^{\Lambda_s} \int_s^T e^{-\Lambda_u} |\varphi_u| du ds\right]\\
&=\mathbb{E}\left[ \lim_{s \to T, s<T} [e^{\Lambda_s} \int_s^T e^{-\Lambda_u} |\varphi_u| du] - \int_0^T e^{-\Lambda_u} |\varphi_u| du + \int_0^T e^{\Lambda_s} e^{-\Lambda_s} |\varphi_s| ds \right]\\
&<+\infty,
\end{align*}
where we have used the estimate $e^{-\Lambda_u} \leq e^{-\Lambda_s}$ for $u\geq s$. We now turn to the definition of the $Z$ process in the equation. Consider the $L^2$ martingale $\hat L$ defined as:
$$ \hat L_t:=\mathbb{E}\left[ \int_0^T e^{-\Lambda_s } \varphi_s ds \vert \mathcal{F}_t\right], \quad t\in [0,T].$$
By the martingale representation theorem, there exists a process $\hat Z$ in $\mathbb{H}^2(\mathbb{R}^d)$ such that $\hat L_t=\hat L_0+\int_0^t \hat Z_s dW_s$. Now let $Z_t:=-e^{\Lambda_t} \hat Z_t$ and $L_t:=\int_0^t Z_s dW_s$ which is a local martingale. With this definition, it is clear that the pair $(Y,Z)$ has the dynamics:
$$ dY_t=(\varphi_t +\lambda_t Y_t) dt +Z_t dW_t, \quad t\in [0,T].$$
Note that a priori $\int_0^\cdot Z_s dW_s$ is only a local martingale.
From the equation, there exists a constant $C>0$ such that
$$ \mathbb{E}\left[\sup_{t\in [0,T]} \left|\int_0^t Z_s dW_s\right|\right] \leq C \left(2\mathbb{E}[\sup_{t\in [0,T]} |Y_t|] + T\|\varphi\|_{\infty} + \mathbb{E}\left[\int_0^T |\lambda_s Y_s| ds\right]\right) <\infty,$$
since by definition $Y$ is bounded.
Hence $Z$ is an element of $\mathbb{H}^1(\mathbb{R}^d)$ by Burkholder-Davis-Gundy's inequality.
Finally note that this argument provides uniqueness of the solution since we have characterized any
solution via the process $\tilde Y$.
\end{proof}

\begin{remark}
Up to a Girsanov transformation, the previous result can be generalized to equations of the form:
$$Y_t=A- \int_t^T (\varphi_s + \sigma_t Z_t-\lambda_s Y_s) ds -\int_t^T Z_s dW_s; \quad t\in [0,T],$$
$$Y_t=A- \int_t^T (\varphi_s + \sigma_t Z_t +\lambda_s Y_s) ds -\int_t^T Z_s dW_s; \quad t\in [0,T],$$
where $\sigma:=(\sigma_t)_{t\in [0,T]}$ is any bounded predictable process. In particular, our results contain the motivating example \eqref{mjareq:proto} from \cite{mjarElKaroui_redbook}.
\end{remark}

\section{A class of non-linear equations}
\label{mjarsection:nonlinear}

From the results of Section \ref{mjarsection:affine} it appears clearly that there is no hope to provide a general theory for equations of the form \eqref{mjareq:proto} with a non-integrable coefficient $\lambda$. However, motivated by financial applications, we need to prove that the particular equation \eqref{mjareq:proto} with $f(x):=\alpha^{-1}(1-e^{-\alpha x})$ admits a unique solution if and only if $Y_T=0$. In addition, in order to provide a complete answer to the financial problem associated to this equation, we need to prove that the process $Y$ is bounded and that the martingale $\int_0^\cdot Z_s dW_s$ is a BMO-martingale (whose definition will be recalled below). This section is devoted to the study of a class of equations which generalizes this particular case. We start with a generalization of Proposition \ref{mjarprop:aff-}.

\begin{prop}
\label{mjarprop:expo1}
Let $\varphi$ be an element of $\mathbb{H}^1(\mathbb{R})$
and $A$ in $L^1$. Let $f:\mathbb{R} \to \mathbb{R}$ be an increasing (respectively decreasing) map with $f(0)=0$. The BSDE
\begin{equation}
\label{mjareq:nonlin1}
Y_t=A-\int_t^T [\varphi_s +\lambda_s f(Y_s)] ds -\int_t^T Z_s dW_s, \quad t\in [0,T]
\end{equation}
admits no solution if $A\not\equiv 0$.
\end{prop}
\begin{proof}
The proof follows the lines of the one of Proposition \ref{mjarprop:aff-} and of Proposition \ref{mjarprop:affine1}.
\end{proof}

\begin{remark}
Note that the previous result does not contradict the conclusion of Proposition \ref{mjarprop:affdet-} in the deterministic setting, since according to Remark \ref{mjarrk:34} the assumption of (i) in Proposition \ref{mjarprop:affdet-} on $\lambda$ is not compatible with the $\mathbb{H}^1(\mathbb{R})$-requirement of Proposition \ref{mjarprop:expo1}.
\end{remark}

The following lemma will be of interest for proving the main result of this section.

\begin{lemma}
\label{mjarlemma:expo}
Let $f:\mathbb{R} \to \mathbb{R}$ satisfying $f(0)=0$, $f$ is non-decreasing and $f(x)- x\leq 0, \; \forall x\in \mathbb{R}$. Then the equation
\begin{equation}
\label{mjareq:sansphi}
Y_t=0-\int_t^T \lambda_s f(Y_s) ds -\int_t^T Z_s dW_s, \quad t\in [0,T]
\end{equation}
admits $(0,0)$ as unique solution.
\end{lemma}

\begin{proof}
It is clear that $(0,0)$ solves \eqref{mjareq:sansphi}. Let $(Y,Z)$ be any solution and $\tilde Y:=e^{-\Lambda} Y$. It holds that $\tilde Y_T=0$ and that
$$ d\tilde Y_t =  \lambda_t e^{-\Lambda_t} (-Y_t+f(Y_t)) dt + e^{-\Lambda_t} Z_t dW_t.$$
Since $  f(x)-x \leq 0$, for all $x\in  \mathbb R$, $\tilde{Y}_t \geq 0.$
Hence by definition, $Y_t\geq 0, \; \forall t\in [0,T]$, $\mathbb{P}$-a.s. From Equation \eqref{mjareq:sansphi}, since $Y\geq 0$ we have that $f(Y_t)\geq0$ which implies that
$$ Y_t =0-\mathbb{E}\left[\int_t^T \lambda_s f(Y_s) ds \vert \mathcal{F}_t \right]  \leq 0, \quad \forall t\in [0,T], \; \mathbb{P}-a.s.$$
As a consequence $Y_t= 0$ for all $t$, $\mathbb{P}$-a.s. which in turn gives $Z = 0$ (in $\mathbb{H}^1(\mathbb{R}^d)$), which concludes the proof.
\end{proof}

We now consider a class of nonlinear BSDEs.

\begin{theorem}
\label{mjarth:expo2}
Let $\varphi$ be a non-negative bounded predictable process and $f:\mathbb{R} \to \mathbb{R}$ a continuously differentiable map satisfying: $f(0)=0$, $f$ is non-decreasing, there exists $\delta>0$ such that
$$f(x)-x\leq 0, \; \forall x\in \mathbb{R} \textrm{ and }  f'(x) \geq \delta, \; \forall x \leq 0.$$
Assume in addition that $\mathbb{E}[\Lambda_t]<+\infty, \; \forall t<T$. Then the BSDE
\begin{equation}
\label{mjareq:expo1}
Y_t=A-\int_t^T (\varphi_s +\lambda_s f(Y_s)) ds -\int_t^T Z_s dW_s, \quad t\in [0,T].
\end{equation}
admits a solution  if and only if $A \equiv 0$. In that case, the solution is unique, $Y$ is bounded and $\int_0^\cdot Z_s dW_s$ is a BMO-martingale, that is:
$$ \mathrm{esssup}_{\tau \in \mathcal{T}}\mathbb{E}\left[\int_\tau^T \|Z_s\|^2 ds \vert \mathcal{F}_\tau \right] <\infty,$$
where we recall that $\mathcal{T}$ denotes the set of stopping time smaller or equal to $T$.
\end{theorem}

\begin{proof}
We have seen in Proposition \ref{mjarprop:expo1} that the only possible value for $A$ to admit a solution is $0$. From now on, we assume that $A \equiv 0$.\\\\
\textbf{Step 1: some estimates\\}
We start with some estimates on the (possible) solution to the BSDE. Assume that there exists a solution  $(Y,Z)$ to Equation \eqref{mjareq:expo1}.
Since $\varphi$ is non-negative, $(Y,Z)$ is a sub-solution to the BSDE:
$$ \mathcal{Y}_t=0-\int_t^T \lambda_s f(\mathcal{Y}_s) ds-\int_t^T \mathcal{Z}_s dW_s, \quad t\in [0,T],$$
which admits $(0,0)$ as unique solution by Lemma \ref{mjarlemma:expo}.
Indeed, this sub-solution property is classical for ($L^2$) Lipschitz BSDE and follows from the comparison theorem. However, here the BSDE \eqref{mjareq:sansphi} is not Lipschitz due to the unboundedness of $\lambda$. In our context the result can be proved explicitly. Since $f(0)=0$ the BSDE \eqref{mjareq:expo1} can be written as\footnote{by: $f(x)-f(y)=(x-y) \int_0^1 f'(y+\theta (x-y)) d\theta$, $(x,y) \in \mathbb{R}^2$}
\begin{equation}
\label{mjareq:delta}
Y_t =0-\int_t^T \tilde{\lambda}_sY_s ds -\int_t^T Z_s dW_s - \int_t^T \varphi_s ds, \quad t\in [0,T],
\end{equation}
with $\tilde \lambda_t:=\lambda_t \int_0^1 f'(\theta Y_t) d\theta$ which is non-negative.
Following the lines of Proposition \ref{mjarprop:affine1} with $\lambda$ replaced by $\tilde\lambda$, and using the non-negativity of $\varphi$, we get that
$$ Y_t =-\mathbb{E}\left[\int_t^T e^{-\int_t^s \tilde{\lambda}_u du} \varphi_s ds \vert \mathcal{F}_t\right] \leq 0. $$
From the non-positivity of $Y$, we can deduce that $\tilde \lambda \geq \delta \lambda$ from which we get that
$$Y_t\geq-\mathbb{E}\left[\int_t^T e^{-\delta \int_t^s \lambda_u du} \varphi_s ds \vert \mathcal{F}_t\right], \quad t\in [0,T].$$
To summarize, we have proven that
\begin{equation}
\label{mjareq:expoesti}
-(T-t) \|\varphi\|_\infty \leq -\mathbb{E}\left[\int_t^T e^{-\delta\int_t^s \lambda_u du} \varphi_s ds \vert \mathcal{F}_t\right] \leq Y_t \leq 0, \quad \forall t \in [0,T], \; \mathbb{P}-a.s..
\end{equation}
We now prove that the process $\int_0^\cdot Z_s dW_s$ is a BMO-martingale. Let $\tau$ be any stopping time such that $\tau\leq T$. By It\^o's formula, we have that
$$ |Y_\tau|^2 = 0-2 \int_\tau^T \varphi_s Y_s ds - 2 \int_\tau^T Y_s Z_s dW_s -\int_\tau^T \|Z_s\|^2 ds -2\int_\tau^T Y_s f(Y_s) \lambda_s ds.$$
Since $Y$ is bounded and $Z$ is an element of $\mathbb{H}^{1}(\mathbb{R}^d)$, the stochastic integral process is a true martingale, and since $Y$ is non-positive, the last term of the previous expression is non-positive. As a consequence, it holds that
$$ \mathbb{E}\left[\int_\tau^T \|Z_s\|^2 ds \vert \mathcal{F}_\tau \right] \leq -2 \mathbb{E}\left[\int_\tau^T \varphi_s Y_s ds \vert \mathcal{F}_\tau \right] \leq 2 T^2\|\varphi\|_\infty^2,$$
So the claim is proved.\\\\
\textbf{Step 2: existence\\}
Now, we prove the existence of a solution for the BSDE \eqref{mjareq:expo1}. For any positive integer $n$, we set $\lambda_\cdot^n:=\lambda_\cdot \wedge n$, $\tilde{f}(x):=f(x) \textbf{1}_{\{[-T \|\varphi\|_\infty, 0]\}}(x) +f(-T \|\varphi\|_\infty) \textbf{1}_{\{(-\infty,-T \|\varphi\|_\infty]\}}(x)$, and $(Y^n,Z^n)$ the unique (classical) solution in $\mathbb{S}^2 \times \mathbb{H}^2(\mathbb{R}^d)$ to the BSDE
\begin{equation}
\label{mjareq:Yapproxim}
Y_t^n = 0-\int_t^T (\varphi_s +\tilde f(Y_s^n) \lambda_s^n) ds -\int_t^T Z_s^n dW_s, \quad t\in [0,T].
\end{equation}
It is clear that this equation admits a unique solution since $\tilde f$ is Lipschitz continuous and $\lambda^n$ is bounded. In addition, by definition, $\tilde f(Y_s^n) \leq 0$, and so $Y_t^n \geq -\|\varphi\|_{\infty} (T-t)$. Thus $(Y^n,Z^n)$ solves the same equation with $\tilde f$ replaced by $\hat f(x):=f(x) \textbf{1}_{\{x\geq 0\}}$. Note that $\hat f(x) \leq x$ for any $x$ in $\mathbb{R}$. Since $\varphi$ is non-negative, $Y^n$ is a classical sub-solution to the BSDE \eqref{mjareq:sansphi} with $f$ replaced by $\hat f$, and so by Lemma \ref{mjarlemma:expo} we deduce that $Y_t^n\leq 0$. Thus
\begin{equation}
\label{mjareq:temp12}
|Y_t^n| \leq (T-t) \|\varphi\|_\infty, \quad \forall t \in [0,T], \; \mathbb{P}-a.s..
\end{equation}
Hence we can re-write Equation \eqref{mjareq:Yapproxim} as:
\begin{equation}
\label{mjareq:Yapproximbis}
Y_t^n = 0-\int_t^T (\varphi_s +f(Y_s^n) \lambda_s^n) ds -\int_t^T Z_s^n dW_s, \quad t\in [0,T].
\end{equation}
Repeating the same argument used in the previous step we can prove that
\begin{equation}
\label{mjareq:boundZn}
\sup_{n} \mathbb{E}\left[\int_0^T \|Z_t^n\|^2 dt\right] <\infty.
\end{equation}
By comparison theorem for Lipschitz BSDEs the sequence $(Y^n)_n$ is non-decreasing. Hence it converges pointwise to some element $Y:=\limsup_{n\to \infty} Y^n$. We would like to point out at this stage that by construction $Y$ takes values in $[-T\|\varphi\|_\infty,0]$.
In view of Dini's theorem, to obtain convergence uniformly in time, we need to prove that $Y$ is continuous. This is done in two steps. Fix $0<t_0<T$, $n\geq 1$ and $p,q \geq n$. We show that the sequence $(Y_n \textbf{1}_{[0,t_0]})_n$ is a Cauchy sequence in $\mathbb{S}^2$.
Let $\delta Y:=Y^p-Y^q$, $\delta Z:=Z^p-Z^q$. It\^o's formula gives for every $t\in [0,t_0]$ that
\begin{equation}
\label{mjareq:apriori1}
|\delta Y_t|^2 + \int_t^{t_0} \|Z_s\|^2 ds \leq |\delta Y_{t_0}|^2 -2 \int_t^{t_0} \delta Y_s f(Y_s^q) (\lambda_s^p-\lambda_s^q) ds - 2 \int_t^{t_0} \delta Y_s \delta Z_s dW_s,
\end{equation}
where we have used the fact that $\delta Y_s (f(Y_s^p)-f(Y_s^q))\geq 0$ since $f$ is non-decreasing. From this relation we deduce in particular for $t=0$ that
\begin{equation} \label{mjareq:apriori2} \mathbb{E}\left[\int_0^{t_0}
\|Z_s\|^2 ds\right]  \leq C \mathbb{E}\left[ |\delta Y_{t_0}|^2 +
\int_0^{t_0} |\lambda_s^p-\lambda_s^q| ds \right],
\end{equation}
since $Y^p$ and $Y^q$ are uniformly (in $p,q$) bounded. Taking the supremum over $[0,t_0]$ in Relation \eqref{mjareq:apriori1} leads to
\begin{align*}
&\mathbb{E}[\sup_{t\in [0,t_0]} |\delta Y_t|^2] \\
&\leq C \left(\mathbb{E}[|\delta Y_{t_0}|^2] +\mathbb{E}\left[\int_0^{t_0} |\delta Y_s f(Y_s^q)| |\lambda_s^p-\lambda_s^q| ds + \sup_{t\in [0,t_0]} \left|\int_t^{t_0} \delta Y_s \delta Z_s dW_s\right|\right]\right)\\
&\leq C \left(\mathbb{E}[|\delta Y_{t_0}|^2] +\mathbb{E}\left[\int_0^{t_0} |\lambda_s^p-\lambda_s^q| ds\right] + \mathbb{E}\left[\left(\int_0^{t_0} |\delta Y_s|^2 \|\delta Z_s\|^2 ds\right)^{1/2}\right]\right)\\
&\leq C \left(\mathbb{E}[|\delta Y_{t_0}|^2] +\mathbb{E}\left[\int_0^{t_0} |\lambda_s^p-\lambda_s^q| ds\right]\right) + \frac12 \mathbb{E}\left[\sup_{t\in [0,t_0]} |\delta Y_t|^2\right] +\frac{C^2}{2} \mathbb{E}\left[\int_0^{t_0} \|\delta Z_s\|^2 ds\right],
\end{align*}
where we have used the fact that $|\delta Y_s f(Y_s^q)|$ is bounded uniformly in $p,q$, the
Burkholder inequality and the inequality $C a b \leq \frac12 a^2+
\frac{C^2 b^2}{2}$. Combining the previous estimate with Estimate
\eqref{mjareq:apriori2} proves that
$$ \mathbb{E}[\sup_{t\in [0,t_0]} |\delta Y_t|^2] \leq C \left(\mathbb{E}[|\delta Y_{t_0}|^2] +\mathbb{E}\left[\int_0^{t_0} |\lambda_s^p-\lambda_s^q| ds\right]\right),$$
where $C$ does not depend on $p,q$. Recalling the definition of $\delta Y=Y^p-Y^q$ it follows that
$$ \lim_{n\to \infty} \sup_{p,q\geq n} \mathbb{E}[\sup_{t\in [0,t_0]} |\delta Y_t|^2] \leq C\lim_{n\to \infty} \left(\mathbb{E}[|Y_{t_0}^n-Y_{t_0}|^2] +\mathbb{E}\left[\int_0^{t_0} |\lambda_s^n-\lambda_s| ds\right]\right) =0,$$
by Lebesgue's dominated convergence Theorem (since $\mathbb{E}[\Lambda_{t_0} ]<\infty$)\footnote{Here we did not use the classical a priori estimates for Lipschitz BSDEs since they would lead to an estimate of the form $\mathbb{E}\left[\int_0^{t_0} |\lambda_s^p-\lambda_s^q|^2 ds\right]$ which is not compatible with our $L^1$ assumption: $\mathbb{E}[\Lambda_{t_0} ]<\infty$.}. Hence $(Y^n \textbf{1}_{[0,t_0]})_n$ is a Cauchy sequence in $\mathbb{S}^2$ which thus converges to $Y \textbf{1}_{[0,t_0]}$. So $Y$ is continuous on $[0,t_0]$ for any $t_0<T$. It remains to prove that $Y$ is continuous at $T$. Let $\varepsilon>0$. By Inequality \eqref{mjareq:temp12} it holds that
\begin{align*}
|Y_{T-\varepsilon}| = \lim_{n \to \infty} |Y^n_{T-\varepsilon}| \leq \varepsilon \|\varphi\|_\infty,
\end{align*}
proving that $Y$ is continuous at $T$. Hence, $(Y^n)_n$ is a non-decreasing sequence of continuous bounded processes converging to a continuous process $Y$, thus by Dini's Theorem, $(Y_n)_n$ converges in $\mathbb{S}^2$ to $Y$.\\\\
\noindent
We now prove that $Y$ together with a suitable process $Z$ solves the BSDE \eqref{mjareq:expo1}. To this end, we aim at applying \cite[Theorem 1]{mjarBarlow_Protter}. We have obtained already that $\lim_{n\to\infty} \mathbb{E}[\sup_{t\in [0,T]} |Y_t^n-Y_t|]=0$.
To satisfy the assumptions of \cite[Theorem 1]{mjarBarlow_Protter}, we have to prove that for every $n$
\begin{equation}
\label{mjareq:BPa1}
\sup_n \mathbb{E}\left[\left(\int_0^T \|Z_s^n\|^2 ds\right)^{1/2}\right] \leq C
\end{equation}
(which by Burkholder's inequality implies that $\mathbb{E}\left[\sup_{t\in [0,T]} \left|\int_0^t Z_s^n dW_s\right|\right] \leq C$ for every $n$) and that
\begin{equation}
\label{mjareq:BPa2}
\sup_n \mathbb{E}\left[\int_0^T \left|\lambda_s^n f(Y_s^n) \right| ds\right] \leq C, \quad \forall n\geq 1,
\end{equation}
since the process $\int_0^\cdot \lambda_s^n f(Y_s^n) ds$ is non-increasing (recall that $Y^n\leq 0$ and the assumptions on $f$).
Relation \eqref{mjareq:BPa1} is a direct consequence of \eqref{mjareq:boundZn}. With this estimate at hand we can deduce Relation \eqref{mjareq:BPa2}. Indeed, using Equation \eqref{mjareq:Yapproximbis} and the uniform estimates on the $Y^n$ obtained above we deduce that
\begin{align*}
\mathbb{E}\left[\int_0^T \vert \lambda_s^n f(Y_s^n) \vert ds\right]&= \mathbb{E}\left[\left\vert \int_0^T \lambda_s^n f(Y_s^n) ds\right\vert\right]\\
&= \mathbb{E}\left[\left\vert Y_0^n + \int_0^T \varphi_s ds +\int_0^T Z_s^n dW_s\right\vert\right]\leq C, \quad n \geq 1,
\end{align*}
where $C$ depends only on $T$ and $\|\varphi\|_\infty$ (and not on $n$).
Thus, by  \cite[Theorem 1]{mjarBarlow_Protter}, $Y$ writes down as $Y_t=A_t+\int_0^t \varphi_s ds +\int_0^t Z_s dW_s$, with $Z \in \mathbb{H}^1(\mathbb{R}^d)$, and
\begin{equation}
\label{mjareq:limA}
\lim_{n\to \infty} \mathbb{E}[\sup_{t\in [0,T]} |A_t-\int_0^t \lambda_s^n f(Y_s^n) ds| ] =0.
\end{equation}
We now identify the process $A$. We proceed in two steps: first we prove that $A_t=\int_0^t f(Y_s) \lambda_s ds$ for $t<T$ and then we prove the relation for $t=T$. Fix $t<T$. We have that
\begin{align*}
&\left|\int_0^t f(Y^n_s) (\lambda_s^n-\lambda_s) ds \right\vert\\
&\leq C \int_0^t |\lambda_s^n-\lambda_s| ds \to_{n\to\infty} 0, \quad \mathbb{P}-a.s.
\end{align*}
by the monotone convergence theorem, since the $Y^n$ are uniformly bounded and $\Lambda_t<\infty$, $\mathbb{P}$-a.s. Hence up to a subsequence,
$$ \lim_{n\to \infty} \left|A_t - \int_0^t f(Y^n_s) \lambda_s ds\right|=0.$$
Recalling that $Y^n\leq Y$, we have that
\begin{align*}
&\left|\int_0^t (f(Y^n_s)-f(Y_s)) \lambda_s ds \right\vert\\
&\leq C \int_0^t |Y^n_s-Y_s| \lambda_s ds \to_{n\to\infty} 0
\end{align*}
where once again we have used monotone convergence Theorem. This leads to
$$ A_t = \int_0^t f(Y_s) \lambda_s ds,\quad \mathbb{P}-a.s.$$
for any $t<T$. The relation for $t=T$ follows from the continuity of $A$ by \eqref{mjareq:limA}. Finally according to Definition \ref{mjarsolution} it remains to prove Relation \eqref{mjareq:BSDEtheo2}. This is done as follows by combining the monotone convergence theorem together with \eqref{mjareq:BPa2} and \eqref{mjareq:limA}:
$$ \mathbb{E}\left[\int_0^T |f(Y_s)| \lambda_s ds\right] = \lim_{t\to T} \mathbb{E}\left[\int_0^t |f(Y_s)| \lambda_s ds\right] = \lim_{t\to T} \mathbb{E}[|A_t|] <\infty.$$
\textbf{Step 3: uniqueness\\}
Assume there exist two solutions $(Y^1,Z^1)$ and $(Y^2,Z^2)$ to the BSDE \eqref{mjareq:expo1}. Then, the difference processes $(\delta Y:=Y^1-Y^2,\delta Z:=Z^1-Z^2)$ satisfies
$$\delta Y_t=0-\int_t^T \lambda_s (f(Y_s^1)-f(Y_s^2)) ds -\int_t^T \delta Z_s dW_s, \quad t\in [0,T].$$
From the existence part we know that both processes $Y^1$ and $Y^2$ are uniformly bounded. As a consequence the mapping $f$ restricted to the set $[-T\|\varphi\|_\infty,0]$ has a non-negative derivative. Hence the equation re-writes as:
$$\delta Y_t=0-\int_t^T \tilde{\lambda}_s \delta Y_s ds -\int_t^T \delta Z_s dW_s, \quad t\in [0,T],$$
where $\tilde{\lambda}_t:=\lambda_t \int_0^1f'(Y_t^2 +\theta (Y_t^1-Y_t^2)) d\theta$ is a non-negative process which satisfies $\int_0^t \tilde\lambda_s ds <\infty$ for $t < T$, $\mathbb{P}$-a.s. and $\int_0^T \tilde\lambda_s ds=\infty$, $\mathbb{P}$-a.s. Similarly to Proposition \ref{mjarprop:affine1} with $\lambda$ replaced with $\tilde \lambda$, we deduce that $(\delta Y,\delta Z)=(0,0)$ is the unique solution.
\end{proof}

\begin{remark}
Note that our previous result is not contained in the theory of monotonic drivers for BSDEs (see e.g. \cite{mjarPardoux_lingrowth,mjarBriandDelyonHuPardouxStoica} or \cite{mjarFanJiang}) where conditions of the form \cite[(H5) and (H1'')]{mjarBriandDelyonHuPardouxStoica} are not satisfied in our setting due to the non-integrability at $T$ of $\Lambda$.
\end{remark}

\section*{Acknowledgments}
The authors are very grateful to Nicole El Karoui and Jean Jacod and for helpful comments and discussions. The financial support of Chaire Markets in transition (FBF) is acknowledged.

\end{document}